\documentclass[12pt]{article}

\usepackage{amsmath,amssymb,amsthm,xspace,amscd}

\theoremstyle{remark}

\newtheorem{remark}{Remark}

\theoremstyle{plain}

\newtheorem{theorem}{Theorem}

\newtheorem{proposition}{Proposition}

\newtheorem{lemma}{Lemma}

\newcommand{\eps}{\varepsilon}

\newcommand{\unitcircle}{{{\mathbb T}^1}}
\newcommand{\R}{{\mathbb R}}
\newcommand{\Z}{{\mathbb Z}}
\newcommand{\DR}{{\rm D}}
\newcommand{\Cr}{{\rm Cr}}
\newcommand{\Dist}{{\rm Dist}}
\newcommand{\Obig}{{\mathcal O}}

\title{Herman's Theory Revisited}

\author{
    K.~Khanin\thanks{Department of Mathematics, University of Toronto
    }
    \and A.~Teplinsky\thanks{Institute of Mathematics, Kiev, Ukraine
    }
    }

\begin{document}

\maketitle
\begin{abstract}
We prove that a $C^{2+\alpha}$-smooth orientation-preserving circle diffeomorphism with rotation number in
Diophantine class $D_\delta$, $0<\delta<\alpha\le1$, is $C^{1+\alpha-\delta}$-smoothly conjugate to a rigid
rotation. We also derive the most precise version of Denjoy's inequality for such diffeomorphisms.
\end{abstract}

\section{Introduction}

An irrational number $\rho$ is said to belong to Diophantine class $D_\delta$ if there exists a constant $C>0$
such that $|\rho-p/q|\ge Cq^{-2-\delta}$ for any rational number $p/q$. The aim of this short note is to present
the new and complete proof of the following
\begin{theorem} Let $T$ be a $C^{2+\alpha}$-smooth orientation-preserving circle
diffeomorphism with rotation number $\rho\in D_\delta$, $0<\delta<\alpha\le1$. Then $T$ is
$C^{1+\alpha-\delta}$-smoothly conjugate to the rigid rotation by angle $\rho$.
\end{theorem}
\noindent (Note, that $C^{2+\alpha}$ with $\alpha=1$ throughout this paper means $C^{2+\text{Lip}}$ rather than
$C^3$.)

This result was first announced in  \cite{SK}. However, the complete proof was never published. Moreover, the argument in \cite{SK} contained a wrong intermediate estimate.

The first global results on smoothness of conjugation with rotations were obtained by M.~Herman \cite{H}. Later
J.-C.~Yoccoz extended the theory to the case of Diophantine rotation numbers \cite{Y}. The case of
$C^{2+\alpha}$-smooth diffeomorphisms was considered by K.~Khanin, Ya.~Sinai \cite{KS,SK} and Y.~Katznelson,
D.~Ornstein \cite{KO1,KO2}.


In the present paper we use a conceptually new approach, which considerably simplifies the proof. We also believe
that this approach will prove useful in other problems involving circle diffeomorphisms.

Let us remark that our result is stronger than the statement proven in \cite{KO1}, although their scope is wider
(namely, we do not consider smoothness higher than $C^3$). It is also sharp, i.e. smoothness of conjugacy higher
than $C^{1+\alpha-\delta}$ cannot be achieved in general settings, as it follows from the examples constructed in
\cite{KO1}.

The paper is self-consistent although requires good understanding of combinatorics of circle homeomorphisms
and Denjoy's theory, for which we refer a reader to the book \cite{S-book}.

\section{Cross-ratio tools}

The {\em cross-ratio} of four pairwise distinct points $x_1, x_2, x_3, x_4$ is
$$
\Cr(x_1,x_2,x_3,x_4)=\frac{(x_1-x_2)(x_3-x_4)}{(x_2-x_3)(x_4-x_1)}
$$
Their {\em cross-ratio distortion} with respect to a strictly increasing function $f$ is
$$
\Dist(x_1,x_2,x_3,x_4;f)=\frac{\Cr(f(x_1),f(x_2),f(x_3),f(x_4))}{\Cr(x_1,x_2,x_3,x_4)}
$$
Clearly,
\begin{equation}\label{eq:Dist=DRtimesDR}
\Dist(x_1,x_2,x_3,x_4;f)=\frac{\DR(x_1,x_2,x_3;f)}{\DR(x_1,x_4,x_3;f)},
\end{equation}
where
$$
D(x_1,x_2,x_3;f)=\frac{f(x_1)-f(x_2)}{x_1-x_2}:\frac{f(x_2)-f(x_3)}{x_2-x_3}
$$
is the {\em ratio distortion} of three distinct points $x_1,x_2,x_3$ with respect to $f$.

In the case of smooth $f$ such that $f'$ does not vanish, both the ratio distortion and the cross-ratio distortion
are defined for points, which are not necessarily pairwise distinct, as the appropriate limits (or, just by
formally replacing ratios $(f(a)-f(a))/(a-a)$ with $f'(a)$ in the definitions above).

Notice that both ratio and cross-ratio distortions are multiplicative with respect to composition: for two
functions $f$ and $g$ we have
\begin{equation}\label{eq:mult-DR}
D(x_1,x_2,x_3;f\circ g)=D(x_1,x_2,x_3;g)\cdot D(g(x_1),g(x_2),g(x_3);f)
\end{equation}
\begin{equation}\label{eq:mult-Dist}
\Dist(x_1,x_2,x_3,x_4;f\circ g)=\Dist(x_1,x_2,x_3,x_4;g)\cdot\Dist(g(x_1),g(x_2),g(x_3),g(x_4);f)
\end{equation}

\begin{proposition}\label{prop:DR}
Let $f\in C^{2+\alpha}$, $\alpha\in[0,1]$, and $f'>0$ on $[A,B]$. Then for any $x_1, x_2, x_3\in[A,B]$ the
following estimate holds:
\begin{equation}\label{eq:DR}
\DR(x_1,x_2,x_3;f)=1+(x_1-x_3)\left(\frac{f''}{2f'}+\Obig(\Delta^\alpha)\right),
\end{equation}
where $\Delta=\max\{x_1,x_2,x_3\}-\min\{x_1,x_2,x_3\}$, and the values of both $f''$ and $f'$ can be taken at any
points between $\min\{x_1,x_2,x_3\}$ and $\max\{x_1,x_2,x_3\}$.
\end{proposition}
\begin{proof}
First of all, it is easy to see why the arguments of $f''$ and $f'$ in the estimate (\ref{eq:DR}) be taken
arbitrarily: $f''(\theta_1)-f''(\theta_2)=\Obig(\Delta^\alpha)$, $f'(\theta_1)-f'(\theta_2)=\Obig(\Delta)$, and
$(f'(\theta))^{-1}=\Obig(1)$.

To prove (\ref{eq:DR}), we need to consider three cases of relative locations of the points.

\noindent Case 1: $x_2$ lies between $x_1$ and $x_3$. It is easy to calculate that
$$
\frac{f(x_1)-f(x_2)}{x_1-x_2}-\frac{f(x_2)-f(x_3)}{x_2-x_3}=(x_1-x_3)\left(\frac{1}{2}f''+\Obig(\Delta^{\alpha})\right),
$$
and (\ref{eq:DR}) follows.

\noindent Case 2: $x_1$ lies between $x_2$ and $x_3$. One can check that
$$
D(x_1,x_2,x_3;f)=1+\left[\frac{x_1-x_3}{x_2-x_3}\left(\frac{f(x_2)-f(x_1)}{x_2-x_1}-
\frac{f(x_1)-f(x_3)}{x_1-x_3}\right)\right]: \frac{f(x_2)-f(x_3)}{x_2-x_3}.
$$
The expression in the round brackets equals $(x_2-x_3)(\frac{1}{2}f''+\Obig(\Delta^{\alpha}))$, so in the square
brackets we have $(x_1-x_3)(\frac{1}{2}f''+\Obig(\Delta^{\alpha}))$.

\noindent Case 3: $x_3$ lies between $x_1$ and $x_2$. Similar to Case~2.
\end{proof}

\begin{proposition}\label{prop:Dist}
Let $f\in C^{2+\alpha}$, $\alpha\in[0,1]$, and $f'>0$ on $[A,B]$. For any $x_1, x_2, x_3, x_4\in[A,B]$ the
following estimate holds:
$$
\Dist(x_1,x_2,x_3,x_4;f)=1+(x_1-x_3)\Obig(\Delta^\alpha)
$$
where $\Delta=\max\{x_1,x_2,x_3,x_4\}-\min\{x_1,x_2,x_3,x_4\}$.
\end{proposition}

\begin{proof}
Follows immediately from Proposition~\ref{prop:DR} due to (\ref{eq:Dist=DRtimesDR}).
\end{proof}

\begin{remark}
While the ratio distortion satisfies an obvious estimate
\begin{equation}\label{eq:log-DR}
\log D(x_1,x_2,x_3;f)=\Obig(x_1-x_3),
\end{equation}
Proposition~\ref{prop:Dist} implies a stronger (for small $\Delta$) estimate for cross-ratio distortion:
\begin{equation}\label{eq:log-Dist}
\log\Dist(x_1,x_2,x_3,x_4;f)=(x_1-x_3)\Obig(\Delta^\alpha)
\end{equation}
\end{remark}

\section{Circle diffeomorphisms}

\subsection{Settings and notations}

For an orientation-preserving homeomorphism $T$ of the unit circle $\unitcircle=\R/\Z$, its {\em rotation number}
$\rho=\rho(T)$ is the value of the limit $\lim_{i\to\infty}L_T^i(x)/i$ for a lift $L_T$ of $T$ from $\unitcircle$
onto $\R$. It is known since Poincare that rotation number is always defined (up to an additive integer) and does
not depend on the starting point $x\in\R$. Rotation number $\rho$ is irrational if and only if $T$ has no periodic
points. We restrict our attention in this paper to this case. The order of points on the circle for any trajectory
$\xi_i=T^i\xi_0$, $i\in\Z$, coincides with the order of points for the {\em rigid rotation}
$$R_{\rho}:\quad\xi\mapsto \xi+\rho\mod 1$$
This fact is sometimes referred to as the {\em combinatorial equivalence} between $T$ and $R_{\rho}$.

We shall use the {\em continued fraction} expansion for the (irrational) rotation number:
\begin{equation}
\rho=[k_1,k_2,\ldots,k_n,\ldots]=\dfrac{1}{k_1+\dfrac{1}{k_2+
\dfrac{1}{\dfrac{\cdots}{k_n+\dfrac{1}{\cdots}}}}}\in(0,1)\label{eq:cont-frac}
\end{equation}
which, as usual,  is understood as a limit of the sequence of {\em rational convergents} $p_n/q_n=[k_1,k_2,\dots,k_n]$. The positive
integers $k_n$, $n\ge1$, called {\em partial quotients}, are defined uniquely for irrational $\rho$. The mutually
prime positive integers $p_n$ and $q_n$ satisfy the recurrent relation $p_n=k_np_{n-1}+p_{n-2}$,
$q_n=k_nq_{n-1}+q_{n-2}$ for $n\ge1$, where it is convenient to define $p_0=0$, $q_0=1$ and $p_{-1}=1$,
$q_{-1}=0$.

Given a circle homeomorphism $T$ with irrational $\rho$, one may consider a {\em marked trajectory} (i.e. the
trajectory of a marked point) $\xi_i=T^i\xi_0\in\unitcircle$, $i\ge0$, and pick out of it the sequence of the {\em
dynamical convergents} $\xi_{q_n}$, $n\ge0$, indexed by the denominators of the consecutive rational convergents
to $\rho$. We will also conventionally use $\xi_{q_{-\!1}}=\xi_0-1$. The well-understood arithmetical properties
of rational convergents and the combinatorial equivalence between $T$ and $R_\rho$ imply that the dynamical convergents
approach the marked point, alternating their order in the following way:
\begin{equation}
\xi_{q_{\!-1}}<\xi_{q_1}<\xi_{q_3}<\dots<\xi_{q_{2m+1}}<\dots<\xi_0<\dots<\xi_{q_{2m}}<\dots<\xi_{q_2}<\xi_{q_0}
\label{eq:dyn-conv}
\end{equation}
We define the $n$th {\em fundamental segment} $\Delta^{(n)}(\xi)$ as the circle arc $[\xi,T^{q_n}\xi]$ if $n$ is
even and $[T^{q_n}\xi,\xi]$ if $n$ is odd. If there is a marked trajectory, then we use the notations
$\Delta^{(n)}_0=\Delta^{(n)}(\xi_0)$, $\Delta^{(n)}_i=\Delta^{(n)}(\xi_i)=T^i\Delta^{(n)}_0$. What is important
for us about the combinatorics of trajectories can be formulated as the following simple

\begin{lemma}\label{lemma:disjoint} For any $\xi\in\unitcircle$ and $0<i<q_{n+1}$ the segments
$\Delta^{(n)}(\xi)$ and $\Delta^{(n)}(T^i\xi)$ are disjoint (except at the endpoints).
\end{lemma}
\begin{proof}
Follows from the combinatorial equivalence of $T$ to $R_{\rho}$ and the following arithmetical fact: the distance
from $i\rho$ to the closest integer is not less than $\Delta_n$ for $0<i<q_{n+1}$ (and equals $\Delta_n$ only for
$i=q_n$, in which case $\Delta^{(n)}(\xi)$ and $\Delta^{(n)}(T^i\xi)$ have a common endpoint $T^{q_n}\xi$).
\end{proof}
In particular, for any $\xi_0$ all the segments $\Delta^{(n)}_i$, $0\le i<q_{n+1}$, are disjoint.

Let us denote $l_{n}=l_n(T)=\max_{\xi}|\Delta^{(n)}(\xi)|=\|T^{q_n}-\text{Id}\|_0$ and
$\Delta_n=l_n(R_\rho)=|q_n\rho-p_n|$. Obviously $l_n, \Delta_n \in (0,1)$ for $n\ge0$, while
$l_{-1}=\Delta_{-1}=1$.

\begin{lemma}\label{lemma:l>Delta}
$l_n\ge\Delta_n$.
\end{lemma}
\begin{proof}
Denote by $\mu$ the unique probability invariant measure for $T$. It follows from the ergodicity of $T^{q_n}$ with
respect to $\mu$ that
$$\int_{\unitcircle}(T^{q_n}(\xi)-\xi)d\mu(\xi)=\rho(T^{q_n}) \mod 1$$
Since $\rho(T^{q_n})=\rho(R_{\rho}^{q_n})=(-1)^n\Delta_n \mod 1$, we have
$$\int_{\unitcircle}|\Delta^{(n)}(\xi)|d\mu(\xi)=\Delta_n,$$ which implies the statement of the lemma.
\end{proof}

It is well known that $\Delta_n\sim\frac{1}{q_{n+1}}$, thus the Diophantine properties of $\rho\in D_\delta$ can
be equivalently expressed in the form:
\begin{equation}\label{eq:Dioph}
\Delta_{n-1}^{1+\delta}=\Obig(\Delta_n)
\end{equation}


\subsection{Denjoy's theory}

The following set of statements essentially summarizes the classical Denjoy theory (see \cite{S-book}), which
holds for any orientation-preserving circle diffeomorphism $T\in C^{1+BV}(\unitcircle)$ with irrational rotation
number $\rho$.

{\bf A}. $\log(T^{q_n})'(\xi_0)=\Obig(1)$.

{\bf B}. There exists $\lambda\in(0,1)$ such that $\frac{|\Delta^{(n+m)}_0|}{|\Delta^{(n)}_0|}=\Obig(\lambda^m)$.

{\bf C}. There exists a homeomorphism $\phi$ that conjugates $T$ to $R_\rho$:
\begin{equation}
\phi\circ T\circ\phi^{-1}=R_\rho
\end{equation}

In order to prove Theorem~1 one has to show that $\phi\in C^{1+\alpha-\delta}(\unitcircle)$ and $\phi'>0$.


\subsection{Denjoy-type inequality}
\label{sect:Denjoy}

The aim of this subsection is to prove the following result that does not require any restrictions on the rotation
number of $T$.
\begin{proposition}[Denjoy-type inequality]\label{prop:Denjoy} Let $T$ be a $C^{2+\alpha}$-smooth, $\alpha\in[0,1]$,
orien\-tation-preserving circle diffeomorphism with irrational rotation number. Then
\begin{equation}\label{eq:Denjoy}
(T^{q_n})'(\xi)=1+\Obig(\eps_n),\quad\text{where}\quad
\eps_n=l_{n-1}^\alpha+\frac{l_n}{l_{n-1}}l_{n-2}^\alpha+\frac{l_n}{l_{n-2}}l_{n-3}^\alpha+\dots+\frac{l_n}{l_0}
\end{equation}
\end{proposition}

\begin{remark}
The inequality (\ref{eq:Denjoy}) can be re-written as
$$
\log(T^{q_n})'(\xi)=\Obig(\eps_n)
$$
\end{remark}

\begin{remark}\label{rem:SKerror}
In the paper \cite{SK} there was a wrong claim (Lemma~12) that one can simply put $\eps_n=l_{n-1}^\alpha$ in
(\ref{eq:Denjoy}). This is not true in the case when $l_{n-1}$ is too small in comparison with $l_{n-2}$, though
comparable with $l_{n}$.
\end{remark}

In order to prove Proposition~\ref{prop:Denjoy}, we introduce the functions
$$
M_n(\xi)=\DR(\xi_0,\xi,\xi_{q_{n-1}};T^{q_n}),\quad \xi\in\Delta_0^{(n-1)};
$$
$$
K_n(\xi)=\DR(\xi_0,\xi,\xi_{q_{n}};T^{q_{n-1}}),\quad \xi\in\Delta_0^{(n-2)},
$$
where $\xi_0$ is arbitrary fixed. The following three exact relations (all of them are easy to check) are crucial
for our proof:

\begin{equation}\label{eq:observ1}
M_n(\xi_0)\cdot M_n(\xi_{q_{n-1}})=K_n(\xi_0)\cdot K_n(\xi_{q_n});
\end{equation}

\begin{equation}\label{eq:observ2}
K_{n+1}(\xi_{q_{n-1}})-1=\frac{|\Delta_0^{(n+1)}|}{|\Delta_0^{(n-1)}|}\left(M_n(\xi_{q_{n+1}})-1\right);
\end{equation}

\begin{equation}\label{eq:observ3}
\frac{(T^{q_{n+1}})'(\xi_0)}{M_{n+1}(\xi_0)}-1=\frac{|\Delta_0^{(n+1)}|}{|\Delta_0^{(n)}|}
\left(1-\frac{(T^{q_{n}})'(\xi_0)}{K_{n+1}(\xi_0)}\right)
\end{equation}

We also need the following lemmas.
\begin{lemma}\label{lemma:Delta-ratio}
$\frac{|\Delta^{(n+m)}_i|}{|\Delta^{(n)}_i|}\sim\frac{|\Delta^{(n+m)}_j|}{|\Delta^{(n)}_j|}$, $0\le j-i<q_{n+1}$.
\end{lemma}
\begin{proof}
Due to (\ref{eq:mult-DR}) and (\ref{eq:log-DR})
\begin{multline*}
\left|\log\frac{|\Delta^{(n+m)}_j|\cdot|\Delta^{(n)}_i|}{|\Delta^{(n+m)}_i|\cdot|\Delta^{(n)}_j|}\right|= |\log
D(\xi_{i+q_{n+m}},\xi_i,\xi_{i+q_n};T^{j-i})|\\
\le\sum_{k=i}^{j-1}|\log
D(\xi_{k+q_{n+m}},\xi_k,\xi_{k+q_n};T)|=\Obig\left(\sum_{k=i}^{j-1}(|\Delta^{(n+m)}_k|+|\Delta^{(n)}_k|)\right)=\Obig(1)
\end{multline*}
since the circle arcs $\Delta^{(n)}_k$, $i\le k<j$, are disjoint due to Lemma~\ref{lemma:disjoint}; the same is
true for the arcs $\Delta^{(n+m)}_k$, $i\le k<j$.
\end{proof}

\begin{lemma}\label{lemma:Delta-l}
$\frac{|\Delta^{(n+m)}_0|}{|\Delta^{(n)}_0|}=\Obig\left(\frac{l_{n+m}}{l_n}\right)$.
\end{lemma}
\begin{proof}
Pick out the point $\xi^*\in\unitcircle$ such that $|\Delta^{(n)}(\xi^*)|=l_{n}$. Due to combinatorics of
trajectories, there exists $0\le i<q_{n+1}+q_{n}$ such that $\xi_{i+q_n}\in\Delta^{(n)}(\xi^*)$, and so
$\Delta^{(n)}(\xi^*)\subset\Delta^{(n)}_i\cup\Delta^{(n)}_{i+q_n}$. It follows that there exists $0\le
i^*<q_{n+1}+2q_n$ (either $i^*=i$ or $i^*=i+q_n$) such that $\Delta^{(n)}_{i^*}\ge l_n/2$, and so
$\frac{|\Delta^{(n+m)}_{i^*}|}{|\Delta^{(n)}_{i^*}|}\le\frac{2l_{n+m}}{l_n}$. The statement now follows from
Lemma~\ref{lemma:Delta-ratio} (since $q_{n+1}+2q_n<3q_{n+1}$, we need to apply Lemma~\ref{lemma:Delta-ratio} at
most three times).
\end{proof}

\begin{lemma}\label{lemma:Dist-T^qn}
\begin{gather*}
\log\Dist(\xi_0,\xi,\xi_{q_{n-1}},\eta;T^{q_n})=\Obig(l_{n-1}^\alpha), \quad \xi,\eta\in\Delta^{(n-1)}_0;\\
\log\Dist(\xi_0,\xi,\xi_{q_n},\eta;T^{q_{n-1}})=\Obig(l_{n}^\alpha), \quad \xi,\eta\in\Delta^{(n-2)}_0
\end{gather*}
\end{lemma}
\begin{proof}
Follows from (\ref{eq:mult-Dist}), (\ref{eq:log-Dist}) and Lemma~\ref{lemma:disjoint} similar to the proof of
Lemma~\ref{lemma:Delta-ratio}.
\end{proof}

\begin{lemma}\label{lemma:l-lambda}
$\frac{l_{n+m}}{l_n}=\Obig(\lambda^m)$.
\end{lemma}
\begin{proof}
Pick out the point $\xi^*\in\unitcircle$ such that $|\Delta^{(n+m)}(\xi^*)|=l_{n+m}$. It is easy to see that
$\frac{l_{n+m}}{l_n}\le\frac{|\Delta^{(n+m)}(\xi^*)|}{|\Delta^{(n)}(\xi^*)|}=\Obig(\lambda^m)$ due to the
statement {\bf A} above.
\end{proof}

\begin{proof}[Proof of Proposition~\ref{prop:Denjoy}]
Since $M_n(\xi)/M_n(\eta)=\Dist(\xi_0,\xi,\xi_{q_{n-1}},\eta;T^{q_n})$ and $K_n(\xi)/K_n(\eta)=$
\linebreak
$\Dist(\xi_0,\xi,\xi_{q_{n}},\eta;T^{q_{n-1}})$, Lemma~\ref{lemma:Dist-T^qn} implies that
$M_n(\xi)/M_n(\eta)=1+\Obig(l_{n-1}^\alpha)$ and $K_n(\xi)/K_n(\eta)=1+\Obig(l_{n}^\alpha)$. Due to the statement
{\bf B} above, the functions $M_n$ and $K_n$ are bounded from above and from below uniformly in $n$. This gives us
\begin{equation}\label{eq:observ1res}
M_n(\xi)=m_n+\Obig(l_{n-1}^\alpha),\qquad K_n(\xi)=m_n+\Obig(l_{n}^\alpha)
\end{equation}
where $m_n^2$ denotes the products in (\ref{eq:observ1}). Due to (\ref{eq:observ2}) and (\ref{eq:observ1res}) we
have
\begin{equation}\label{eq:observ1sub}
m_{n+1}-1=\frac{|\Delta_0^{(n+1)}|}{|\Delta_0^{(n-1)}|}(m_n-1)+\Obig(l_{n+1}^\alpha),
\end{equation}
which is iterated into
\begin{multline*}
m_n-1=\Obig\left(\sum_{k=0}^n l_{n-k}^\alpha
\frac{|\Delta^{(n)}_0|}{|\Delta^{(n-k)}_0|}\frac{|\Delta^{(n-1)}_0|}{|\Delta^{(n-k-1)}_0|}\right)\\
=\Obig\left(l_n^\alpha\sum_{k=0}^n\left(\frac{l_n}{l_{n-k}}\right)^{1-\alpha}\frac{l_{n-1}}{l_{n-k-1}}\right)
=\Obig\left(l_n^\alpha\sum_{k=0}^n(\lambda^{2-\alpha})^k\right)=\Obig(l_n^\alpha)
\end{multline*}
due to Lemmas~\ref{lemma:Delta-l} and~\ref{lemma:l-lambda}. Hence,
\begin{equation}\label{eq:observ2res}
M_n(\xi)=1+\Obig(l_{n-1}^\alpha),\qquad K_n(\xi)=1+\Obig(l_{n}^\alpha)
\end{equation}
Due to (\ref{eq:observ3}) and (\ref{eq:observ2res}) we have
\begin{equation}\label{eq:observ3sub}
(T^{q_{n+1}})'(\xi_0)-1=\frac{|\Delta_0^{(n+1)}|}{|\Delta_0^{(n)}|}(1-(T^{q_{n}})'(\xi_0))+\Obig(l_n^\alpha)
\end{equation}
which is iterated into
$$
(T^{q_{n}})'(\xi_0)-1=\Obig\left(\sum_{k=0}^n l_{n-k-1}^\alpha\frac{|\Delta^{(n)}_0|}{|\Delta^{(n-k)}_0|}\right)
$$
The statement of the proposition now follows from Lemma~\ref{lemma:Delta-l}.
\end{proof}

\begin{remark}\label{rem:4}
Due to Lemma~\ref{lemma:l-lambda}, $\eps_n=\Obig(\lambda^{\alpha n})$ for $0\le\alpha<1$ and
$\eps_n=\Obig(n\lambda^n)$ for $\alpha=1$, so $\eps_n$ decays exponentially for $\alpha>0$.
\end{remark}

\subsection{Exponential bound on $k_{n+1}\eps_n$}

Let $r(n+m,n)$, $m\ge0$, be the number of indices $0\le
i<q_{n+m+1}$ such that $\Delta_i^{(n+m)}\subset\Delta_0^{(n)}$. It
is easy to see that $r(n,n)=1$, $r(n+1,n)=k_{n+2}$,
$r(n+m,n)=r(n+m-1,n)k_{n+m+1}+r(n+m-2,n)$ for $m\ge2$.
\begin{lemma}\label{lemma:l_m/l_n} There exists a constant $C>0$ such that
\begin{equation}\label{eq:l_m/l_n}
\frac{l_n}{l_{n+m}}\ge r(n+m,n)\left(1-C\sum_{s=n+1}^{n+m}k_{s+1}\eps_s\right)
\end{equation}
\end{lemma}
\begin{proof}
If $\Delta_i^{(n+m)}$ and $\Delta_j^{(n+m)}$, $0\le
i,j<q_{n+m+1}$, are contained in $\Delta_0^{(n)}$, then
$$
\log\frac{|\Delta_i^{(n+m)}|}{|\Delta_j^{(n+m)}|}=\Obig\left(\sum_{s=n+1}^{n+m}k_{s+1}\eps_s\right)
$$
due to the combinatorics of dynamical partitions and Proposition~\ref{prop:Denjoy}. (One of the segments
$\Delta_i^{(n+m)}$ and $\Delta_j^{(n+m)}$ is mapped onto another by a composition of no more than $k_{n+2}$ maps
$T^{q_{n+1}}$, no more than $k_{n+3}$ maps $T^{q_{n+2}}$, \dots, and no more than $k_{n+m+1}$ maps $T^{q_{n+m}}$.)
It follows that there exists $C>0$ such that
$$
|\Delta_0^{(n)}|\ge r(n+m,n)|\Delta_i^{(n+m)}|\left(1-C\left(\sum_{s=n+1}^{n+m}k_{s+1}\eps_s\right)\right)
$$
for any fixed $0\le i<q_{n+m+1}$ such that $\Delta_i^{(n+m)}\subset\Delta_0^{(n)}$. Now we choose $\xi_0$ in such
a way that $|\Delta_i^{(n+m)}|=l_{n+m}$ and obtain (\ref{eq:l_m/l_n}).
\end{proof}

\begin{proposition}\label{prop:k_n+1eps_n}
For any chosen $\lambda_0\in(\lambda^{\alpha-\delta},1)$, the following asymptotics hold:
$$
k_{n+1}\eps_n=\Obig(\lambda_0^n)
$$
\end{proposition}
\begin{proof}
Consider the sequence $n_i$, $i\ge0$, of all indices $n$ such that $k_{n+1}\eps_n>\lambda_0^n$, and assume it to
be infinite. Similarly to proof of Lemma~\ref{lemma:l_m/l_n}, we have
$|\Delta^{(n)}_{q_{n-1}+kq_n}|\ge|\Delta^{(n)}_{q_{n-1}}|(1-Ck\eps_n)$. Choosing $k^*_{n_i+1}\le k_{n_i+1}$ in
such a way that $k^*_{n_i+1}\eps_{n_i}>\lambda_0^{n_i}$ but $Ck^*_{n_i+1}\eps_{n_i}\le\frac{1}{2}$ (it is possible
for large enough $i$ since both $\lambda_0^n$ and $\eps_n$ decay exponentially), we achieve
$|\Delta^{(n_i-1)}_{0}|\ge\sum_{k=0}^{k^*_{n_i+1}-1}|\Delta^{(n_i)}_{q_{n_i-1}+kq_{n_i}}|\ge
\frac{1}{2}k^*_{n_i+1}|\Delta^{(n_i)}_{q_{n_i-1}}|$. With $\xi_0$ such that
$|\Delta^{(n_i)}_{q_{n_i-1}}|=l_{n_i}$, this implies
\begin{equation}\label{eq:Lambda/eps}
\frac{l_{n_i-1}}{l_{n_i}}>\frac{\lambda_0^{n_i}}{2\eps_{n_i}}
\end{equation}
From the equality $\eps_{n_i}=l_{n_i-1}^\alpha+\frac{l_{n_i}}{l_{n_i-1}}\eps_{n_i-1}$, in view of
(\ref{eq:Lambda/eps}) we get $\eps_{n_i}(1-2\eps_{n_i-1}\lambda_0^{-n_i})<l_{n_i-1}^\alpha$. Since
$\eps_n\lambda_0^{-n}$ decays exponentially (see Remark~\ref{rem:4}), this proves that
$\eps_{n_i}=\Obig(l_{n_i-1}^\alpha)$. Hence, (\ref{eq:Lambda/eps}) implies
\begin{equation}\label{eq:l/l1}
l_{n_i}=\Obig(l_{n_i-1}^{1+\alpha}\lambda_0^{-n_i})
\end{equation}
Due to Lemma~\ref{lemma:l_m/l_n}, $\frac{l_{n_{i-1}}}{l_{n_i-1}}\ge
r(n_i-1,n_{i-1})\left(1-C\sum_{s=n_{i-1}+1}^{n_i-1}\lambda_0^s\right)\ge\frac{1}{2}r(n_i-1,n_{i-1})$ for large
enough $i$, so
\begin{equation}\label{eq:l/l2}
l_{n_i-1}=\Obig\left(\frac{l_{n_{i-1}}}{r(n_i-1,n_{i-1})}\right)
\end{equation}
The estimate (\ref{eq:l/l1}) and Lemma~\ref{lemma:l-lambda} imply
$l_{n_i}=\Obig(l_{n_i-1}^{1+\delta+\kappa}\lambda^{(\alpha-\delta-\kappa)n_i}\lambda_0^{-n_i})$ for any
$\kappa\in(0,\alpha-\delta)$. Having taken $\kappa$ so small that $\lambda^{\alpha-\delta-\kappa}<\lambda_0$ and
using (\ref{eq:l/l2}), we achieve
\begin{equation}\label{eq:l/lmain}
l_{n_i}\le\left(\frac{l_{n_{i-1}}}{r(n_i-1,n_{i-1})}\right)^{1+\delta+\kappa}
\end{equation}
for large enough $i$.

Now we start to use the Diophantine properties of rotation number $\rho$. We have
$\Delta_n=r(n+m,n)\Delta_{n+m}+r(n+m-1,n)\Delta_{n+m+1}$, so
$\Delta_{n_{i-1}}=\Obig(r(n_i-1,n_{i-1})\Delta_{n_i-1})$. The property (\ref{eq:Dioph}) implies
$\Delta_{n_i-1}^{1+\delta+\kappa/2}=\Obig(\Delta_{n_i}\Delta_{n_i-1}^{\kappa/2})$, hence
\begin{equation}\label{eq:D/Dmain}
\left(\frac{\Delta_{n_{i-1}}}{r(n_i-1,n_{i-1})}\right)^{1+\delta+\kappa/2}\le\Delta_{n_i}
\end{equation}
for large enough $i$.

Notice, that $0<\Delta_n\le l_n<1$ for all $n$. It follows from (\ref{eq:l/lmain}) and (\ref{eq:D/Dmain}) that
\begin{equation}
\frac{\log l_{n_i}}{\log \Delta_{n_i}}\ge\frac{1+\delta+\kappa}{1+\delta+\kappa/2}\cdot \frac{\log
l_{n_{i-1}}-\log r(n_i-1,n_{i-1})}{\log \Delta_{n_{i-1}}-\log r(n_i-1,n_{i-1})}\ge K\frac{\log l_{n_{i-1}}}{\log
\Delta_{n_{i-1}}}
\end{equation}
for large enough $i$, with $K=\frac{1+\delta+\kappa}{1+\delta+\kappa/2}>1$, so the sequence $\gamma_i=\frac{\log
l_{n_i}}{\log \Delta_{n_i}}>0$ is unbounded. But $\gamma_i\le1$ due to Lemma~\ref{lemma:l>Delta}. This
contradiction proves that $k_{n+1}\eps_n\le\lambda_0^n$ for large enough $n$.
\end{proof}

\subsection{$C^1$-smoothness of $\phi$}

There is more than one way to derive $C^1$-smoothness of the conjugacy from the convergence of the series
$\sum_nk_{n+1}\eps_n$. We will construct the continuous density $h:\unitcircle\to(0,+\infty)$ of the invariant
probability measure for $T$, as in \cite{SK}.

\begin{proposition}\label{prop:C^1}
$\phi$ is a $C^1$-smooth diffeomorphism.
\end{proposition}

\begin{proof}
Consider arbitrary trajectory $\Xi=\{\xi_i,i\in\Z\}$, $\xi_i=T^i\xi_0$, and define a function $\gamma:\Xi\to\R$
by use of the following recurrent relation:
$$
\gamma(\xi_0)=0;\quad\gamma(\xi_{i+1})=\gamma(\xi_{i})-\log T'(\xi_{i}),\quad i\in\Z
$$
As soon as $\xi_j\in\Delta_i^{(n)}$, $j>i$, we have
$$
\gamma(\xi_i)-\gamma(\xi_j)=\Obig\left(\eps_n+\sum_{s=n+1}^{+\infty}k_{s+1}\eps_s\right)=\Obig(\lambda_0^n)\to0,\quad
n\to+\infty
$$
due to combinatorics of a trajectory and Proposition~\ref{prop:Denjoy}. It follows that $\gamma\in C(\Xi)$. Since
$\Xi$ is dense in $\unitcircle$, the function $\gamma$ can be extended continuously onto $\unitcircle$. The
function $h(\xi)=e^{\gamma(\xi)}\left(\int_{\unitcircle} e^{\gamma(\eta)}d\eta\right)^{-1}$ is continuous and
positive on $\unitcircle$, satisfies the {\em homological equation}
\begin{equation}\label{eq:homol}
h(T\xi)=\frac{1}{T'(\xi)}h(\xi),\quad\xi\in\unitcircle,
\end{equation}
and $\int_\unitcircle h(\xi)d\xi=1$. It is easy to check that the $C^1$-smooth diffeomorphism
$$
\phi(\xi)=\int_{\xi_0}^\xi h(\eta)d\eta,\quad\xi\in\unitcircle
$$
conjugates $T$ to $R_\rho$.
\end{proof}

\subsection{$C^{\alpha-\delta}$-smoothness of $h$}

A straightforward corollary of Proposition~\ref{prop:C^1} is that $l_n\sim\Delta_n$.

\begin{lemma}\label{lemma:eps_n}
$\eps_n=\Obig(\Delta_n^{\frac{\alpha}{1+\delta}})$.
\end{lemma}
\begin{proof}
The Diophantine condition $\Delta_{n-1}^{1+\delta}=\Obig(\Delta_n)$ implies that
\begin{multline*}
\eps_n=\Obig\left(\sum_{m=0}^{n}\frac{\Delta_n}{\Delta_{n-m}}\Delta_{n-m-1}^\alpha\right)=
\Obig\left(\Delta_n\sum_{m=0}^{n}\Delta_{n-m}^{\frac{\alpha}{1+\delta}-1}\right)=\\
\Obig\left(\Delta_n^{\frac{\alpha}{1+\delta}}
\sum_{m=0}^{n}\left(\frac{\Delta_{n}}{\Delta_{n-m}}\right)^{\frac{1-\alpha+\delta}{1+\delta}}\right)=
\Obig(\Delta_n^{\frac{\alpha}{1+\delta}}),
\end{multline*}
since $\frac{\Delta_{n}}{\Delta_{n-m}}=\Obig(\lambda^m)$ is exponentially small in $m$.
\end{proof}

\begin{remark}
Since $k_{n+1}\Delta_n<\Delta_{n-1}=\Obig(\Delta_n^{\frac{1}{1+\delta}})$, Lemma~\ref{lemma:eps_n} implies that
$$
k_{n+1}\eps_n=\Obig(\Delta_{n}^{\frac{\alpha-\delta}{1+\delta}})=\Obig(\Delta_{n-1}^{\alpha-\delta})
$$
This improves Proposition~\ref{prop:k_n+1eps_n} a posteriori.
\end{remark}

\begin{proposition}
$h\in C^{\alpha-\delta}(\unitcircle)$.
\end{proposition}
\begin{proof}
Consider two points $\xi_0,\xi\in\unitcircle$ and such $n$ that $\Delta_n\le|\phi(\xi)-\phi(\xi_0)|<\Delta_{n-1}$.
Let $k$ be the greatest positive integer such that $|\phi(\xi)-\phi(\xi_0)|\ge k\Delta_n$. (It follows that $1\le
k\le k_{n+1}$.) Due to the combinatorics of trajectories, continuity of $h$ and the homologic equation
(\ref{eq:homol}), we have
$$
\left|\log h(\xi)-\log h(\xi_0)\right|=\Obig\left(k\eps_n+\sum_{s=n+1}^{+\infty}k_{s+1}\eps_s\right)
$$
The right-hand side here is bounded and so is $h$, hence the same estimate holds for $|h(\xi)-h(\xi_0)|$. By
Lemma~\ref{lemma:eps_n}, we have
$$
k\eps_n=\Obig\left(k^{\alpha-\delta}\Delta_n^{\alpha-\delta}
\left(\frac{k\Delta_n}{\Delta_{n-1}}\right)^{1-\alpha+\delta}\right)=
\Obig((k\Delta_n)^{\alpha-\delta})
$$
and
$$
\sum_{s=n+1}^{+\infty}k_{s+1}\eps_s=\Obig\left(\sum_{s=n+1}^{+\infty}\Delta_{s-1}^{\alpha-\delta}\right)=
\Obig(\Delta_{n}^{\alpha-\delta}),
$$
so $|h(\xi)-h(\xi_0)|=\Obig((k\Delta_n)^{\alpha-\delta})=\Obig(|\phi(\xi)-\phi(\xi_0)|^{\alpha-\delta})=
\Obig(|\xi-\xi_0|^{\alpha-\delta})$.
\end{proof}


\end{document}